\documentclass[a4paper,12pt]{amsart}
\usepackage[utf8]{inputenc}
\usepackage[margin=2.5cm]{geometry}
\usepackage{amsmath}
\usepackage{amsfonts}
\usepackage{amssymb}
\usepackage{amsthm}
\usepackage{tikz-cd} % package for category theory diagrams
\usepackage{pdfpages}
\usepackage{hyperref} % clickable references and toc
\hypersetup{
    citebordercolor = {cyan},
    urlbordercolor = {magenta}
}

\theoremstyle{plain}
\newtheorem*{thm*}{Theorem}
\newtheorem{thm}{Theorem}[section]
\newtheorem{prop}[thm]{Proposition}
\newtheorem{lem}[thm]{Lemma}
\newtheorem{cor}[thm]{Corollary}
\newtheorem{conj}[thm]{Conjecture}

\theoremstyle{definition}

\newtheorem{dfn}[thm]{Definition}
\newtheorem{ex}[thm]{Example}
\newtheorem{rem}[thm]{Remark}
\newtheorem{ntt}[thm]{Notation}

\numberwithin{equation}{section}

\newcommand{\beq}{\begin{equation}}
\newcommand{\eeq}{\end{equation}}

\newcommand{\ZZ}{\mathbb{Z}}

\newcommand{\NN}{\mathbb{N}}

\newcommand{\kk}{\Bbbk}
\newcommand{\Aa}{\mathbb{A}}

\newcommand{\g}{\mathfrak{g}}
\newcommand{\h}{\mathfrak{h}}
\newcommand{\LL}{\mathcal{L}}
\DeclareMathOperator{\ad}{ad}

\newcommand{\mc}{\mathcal}
\newcommand{\mf}{\mathfrak}

\newcommand{\inv}[1]{#1^{-1}}

\DeclareMathOperator{\Ann}{Ann}

\DeclareMathOperator{\Aut}{Aut}
\DeclareMathOperator{\spn}{span}
\DeclareMathOperator{\gr}{gr}
\DeclareMathOperator{\Der}{Der}

\newcommand{\W}{W_{\geq -1}}

\DeclareMathOperator{\Ver}{Ver}

\newcommand{\del}{\partial}
\newcommand\restr[2]{{
  \left.\kern-\nulldelimiterspace % automatically resize the bar with \right
  #1 % the function
  \vphantom{\big|}
  \right|_{#2}
  }}

\title{Enveloping algebras of Krichever-Novikov algebras are not noetherian}
\author{Lucas Buzaglo}
\date{}


\begin{document}

\begin{abstract}
    This work is part of the overarching question of whether it is possible for the universal enveloping algebra of an infinite-dimensional Lie algebra to be noetherian. The main result of this paper is that the universal enveloping algebra of any Krichever-Novikov algebra is not noetherian, extending a result of Sierra and Walton on the Witt (or classical Krichever-Novikov) algebra.
    
    As a subsidiary result, which may be of independent interest, we show that if $\h$ is a Lie subalgebra of $\g$ of finite codimension, then the noetherianity of $U(\h)$ is equivalent to the noetherianity of $U(\g)$.
    
    The second part of the paper focuses on Lie subalgebras of $\W = \Der(\kk[t])$. In particular, we prove that certain subalgebras of $\W$ (denoted by $L(f)$, where $f \in \kk[t]$) have non-noetherian universal enveloping algebras, and provide a sufficient condition for a subalgebra of $\W$ to have a non-noetherian universal enveloping algebra. Furthermore, we make significant progress on a classification of subalgebras of $\W$ by showing that any infinite-dimensional subalgebra must be contained in some $L(f)$ in a canonical way.
\end{abstract}

\maketitle

\section*{Introduction}

Let $\kk$ be a field of characteristic $0$ and let $\overline{\kk}$ be its algebraic closure. For brevity, we say that a $\kk$-algebra is \emph{noetherian} if it is left and right noetherian. Note that the universal enveloping algebra of a Lie algebra is left noetherian if and only if it is right noetherian. Therefore, when referring to universal enveloping algebras there is no need to distinguish between left and right noetherianity.

We are motivated by the long-standing question of whether it is possible for an infinite-dimensional Lie algebra to have a noetherian universal enveloping algebra. This question has been posed by many authors, but its earliest appearance was nearly fifty years ago, in Amayo and Stewart's book on infinite-dimensional Lie algebras \cite[Question 2.7]{AmayoStewart}.

It is widely believed that the answer to this question is ``no" \cite[Conjecture 0.1]{SierraWalton}. The simplest example, an infinite-dimensional abelian Lie algebra, has a universal enveloping algebra isomorphic to a polynomial ring in infinitely many variables, which is clearly non-noetherian.

In \cite{SierraWalton}, Sierra and Walton proved that the universal enveloping algebra of the \emph{Witt algebra} $W = \Der(\kk[t,\inv{t}])$ is not noetherian, providing the first nontrivial example of a non-noetherian enveloping algebra. The question of whether $U(W)$ is noetherian first appeared in print in \cite{DeanSmall}, over twenty years before being answered.

Sierra and Walton's proof also allows us to conclude that the Lie algebra $W_{\geq -1} = \Der(\kk[t])$ has a non-noetherian universal enveloping algebra. An easy consequence of this result is that any infinite-dimensional $\ZZ$-graded simple Lie algebra of polynomial growth has a non-noetherian enveloping algebra \cite[Corollary 0.6]{SierraWalton}. Beyond this, the noetherianity of enveloping algebras of infinite-dimensional Lie algebras remains mysterious. This paper focuses on questions of noetherianity for general Lie algebras of vector fields on affine curves.

We begin by proving the following result, which will be useful throughout the paper, and may be of independent interest.

\begin{prop}[Proposition \ref{prop:finite codimension}]\label{thm:intro finite codimension}
    Let $\h \subseteq \g$ be Lie algebras such that $\h$ has finite codimension in $\g$ (i.e. $\dim_\kk (\g/\h) < \infty)$. Then $U(\h)$ is noetherian if and only if $U(\g)$ is noetherian.
\end{prop}

The proof is achieved by constructing a positive filtration of $U(\g)$ with the property that the associated graded ring is noetherian if $U(\h)$ is noetherian, which is enough to conclude that $U(\g)$ is noetherian.

In Section \ref{sec:affine curves}, we consider \emph{Krichever-Novikov algebras}. A Krichever-Novikov algebra is the Lie algebra of vector fields on an affine curve. Since the Lie algebras considered by Sierra and Walton are examples of Krichever-Novikov algebras, it is natural to ask whether universal enveloping algebras of Krichever-Novikov algebras are noetherian. We show that they are not, providing a second general class of infinite-dimensional Lie algebras with non-noetherian universal enveloping algebras.

\begin{thm}[Theorem \ref{thm:main}]\label{thm:intro main}
    Let $\LL$ be a Krichever-Novikov algebra. Then $U(\LL)$ is not noetherian.
\end{thm}

In order to prove Theorem \ref{thm:intro main} we first note that, by a result from \cite{SierraWalton}, it suffices to show that $\LL$ has a Lie subalgebra with a non-noetherian universal enveloping algebra. We begin by considering the case where $\LL$ is the Lie algebra of vector fields on a nonsingular curve. In this case, we construct an injective Lie algebra homomorphism from a finite-codimensional subalgebra of $W_{\geq -1}$ to $\LL$. By Proposition \ref{thm:intro finite codimension}, this subalgebra has a non-noetherian universal enveloping algebra, therefore proving that $U(\LL)$ is not noetherian.

The injective Lie algebra homomorphism is constructed as follows: let $\pi \colon C \to \Aa^1$ be a dominant morphism, where $C$ is the nonsingular affine curve corresponding to the Krichever-Novikov algebra $\LL$. We show that any vector field $v$ on $\Aa^1$ which vanishes on the branch point locus of $\pi$ extends uniquely to a vector field $\widetilde{v}$ on $C$. The map $v \mapsto \widetilde{v}$ is the injective homomorphism we need.

If $\LL$ is the Lie algebra of vector fields on a singular curve $C$, we require a different argument. We consider the normalisation $\widetilde{C}$ of $C$ and let $\widetilde{\LL}$ be the Lie algebra of vector fields on $\widetilde{C}$. We construct a Lie subalgebra of $\widetilde{\LL}$ which is also a subalgebra of $\LL$ and further prove that the enveloping algebra of this subalgebra is not noetherian. This is enough to conclude that $U(\LL)$ is not noetherian.

At the end of Section \ref{sec:affine curves}, we further show that some related Lie algebras known as \emph{Krichever-Novikov type algebras} (see Definition \ref{def:KN type}) also have non-noetherian enveloping algebras.

After this, in Section \ref{sec:subalgebras}, we turn our attention to arbitrary subalgebras of $W_{\geq -1}$ and make the following conjecture.

\begin{conj}\label{conj:subalgebras}
    Let $\mf{g}$ be an infinite-dimensional Lie subalgebra of $W_{\geq -1}$. Then $U(\mf{g})$ is not noetherian.
\end{conj}

We are not able to prove Conjecture \ref{conj:subalgebras}. However, we make significant progress in understanding the structure of infinite-dimensional subalgebras of $\W$.

We first show that the above conjecture holds for graded subalgebras of $\W$. Furthermore, if $\g$ is a Lie subalgebra of $W_{\geq -1}$ of finite codimension, then $U(\g)$ is not noetherian by Proposition \ref{thm:intro finite codimension}. We therefore focus on (ungraded) infinite-dimensional subalgebras of $W_{\geq -1}$ of infinite codimension.

Veronese subalgebras of $\W$ (see Definition \ref{def:Veronese}) and their Lie subalgebras were the only known infinite-dimensional subalgebras of $W_{\geq -1}$ of infinite codimension. However, we give a new class of such subalgebras which are not, in general, contained in any Veronese. We denote these subalgebras by $L(f)$, where $f \in \kk[t]$. Given $f \in \kk[t]$, the subalgebra $L(f)$ is a $\kk[f]$-submodule of $\W$. We give two proofs that these subalgebras always have non-noetherian universal enveloping algebras. The first one is direct.

In Section \ref{subsec:contained in L(f)} we make partial progress on a classification of subalgebras of $\W$ by attempting to prove the following conjecture.

\begin{conj}\label{conj:intro classification}
    Let $\g$ be an infinite-dimensional subalgebra of $\W$. Then there exists $f \in \kk[t]$ such that $\g$ has finite codimension in $L(f)$.
\end{conj}

Subalgebras of finite codimension in $\W$ have already been classified in \cite{PetukhovSierra}; it is known that they satisfy Conjecture \ref{conj:intro classification}. Therefore, it remains to prove the conjecture for subalgebras of $\W$ of infinite codimension. 

Note that if Conjecture \ref{conj:intro classification} holds, then so does Conjecture \ref{conj:subalgebras}. We tackle Conjecture \ref{conj:intro classification} by introducing two closely related invariants of subalgebras of $\W$. The first one is the \emph{set of ratios} of $\g$, defined as follows:
$$R(\g) = \left\{\frac{w}{u} \mid w\del,u\del \in \g, u \neq 0\right\}.$$
The second invariant is the \emph{field of ratios} $F(\g) = \kk(R(\g))$, the subfield of $\kk(t)$ generated by $R(\g)$.

We first consider the field of ratios, and prove the following result:

\begin{thm}[Theorem \ref{thm:subfield generated by R(g)}]\label{thm:intro field of ratios}
    Let $\g$ be an infinite-dimensional Lie subalgebra of $\W$. Then there exists some $f \in \kk[t] \setminus \{0\}$ such that:
    \begin{enumerate}
        \item The field of ratios $F(\g)$ is generated by $f$.
        \item The subalgebra $\g$ is contained in $L(f)$.
    \end{enumerate}
\end{thm}

We further believe that $\g$ has finite codimension in the subalgebra $L(f)$ from Theorem \ref{thm:intro field of ratios}, but are not able to prove it.

The proof of Theorem \ref{thm:intro field of ratios} goes as follows: by L\"uroth's theorem there exists $f \in \kk(t)$ such that $F(\g) = \kk(f)$. We first show that $f$ can be chosen to be a polynomial, and then deduce that $\g \subseteq L(f)$.

We then move on to the set of ratios $R(\g)$. In particular, we consider the situation where $R(\g)$ is a field:

\begin{thm}[Theorem \ref{thm:S is a field?}]\label{thm:intro equivalent conditions}
    Let $\g$ be an infinite-dimensional Lie subalgebra of $\W$. The following are equivalent:
    \begin{enumerate}
        \item[(1)] There exists a nonzero $f \in \kk[t]$ such that $\g$ has finite codimension in $L(f)$.
        \item[(2)] The set $R(\g)$ is a field.
        \item[(3)] There exists a nonzero $f \in \kk[t]$ such that $R(\g) = \kk(f)$.
    \end{enumerate}
\end{thm}

The proof of Theorem \ref{thm:intro equivalent conditions} follows by applying Theorem \ref{thm:intro field of ratios} and an analogous result to \cite[Proposition 3.19]{PetukhovSierra} for subalgebras of $\W$. As an immediate consequence of Theorem \ref{thm:intro equivalent conditions}, we get:

\begin{cor}[Corollary \ref{cor:R(g) is a field}]
    Let $\g$ be an infinite-dimensional Lie subalgebra of $\W$. If $R(\g)$ is a field then $U(\g)$ is not noetherian.
\end{cor}

If we do not assume that $R(\g)$ is a field, then it is not clear is whether $\g$ has finite codimension in $L(f)$, or whether $U(\g)$ can be noetherian in this case.

We finish the paper by finding a sufficient condition on a subalgebra $\g \subseteq \W$, which we call $(*)$, for $U(\g)$ to be non-noetherian. This condition is defined by analysing the image of $U(\g)$ under a certain homomorphism
$$\Phi \colon U(\W) \to A_1(\kk(y))[\inv{t}].$$
Here, $A_1(\kk(y)) = \kk(y)[t,\del]$ denotes the first Weyl algebra over $\kk(y)$. All known infinite-dimensional subalgebras of $\W$ satisfy $(*)$; in particular, we show that the subalgebras $L(f)$ satisfy $(*)$, providing a second proof of the non-noetherianity of $U(L(f))$. We believe that condition $(*)$ will be useful in future work on Conjecture \ref{conj:subalgebras}.

\smallskip

\noindent \textbf{Acknowledgements:} This work was done as part of the author's PhD research at the University of Edinburgh.

We thank Pavel Etingof for extensive comments and useful discussion on an earlier version of the paper. We also thank Jason Bell for useful comments and discussion.

\section{Lie subalgebras of finite codimension}\label{sec:finite codimension}

If $\h$ is a Lie subalgebra of $\g$ and $U(\g)$ is noetherian, then so is $U(\h)$ \cite[Lemma 1.7]{SierraWalton}, which is a consequence of the faithful flatness of $U(\g)$ over $U(\h)$. It is much less clear when the converse implication holds. The main purpose of this section is to prove that when $\h$ has finite codimension in $\g$, then the noetherianity of $U(\h)$ is equivalent to the noetherianity of $U(\g)$. We will use this result extensively in later sections.

\begin{prop}\label{prop:finite codimension}
    Let $\h \subseteq \g$ be Lie algebras such that $\h$ has finite codimension in $\g$. Then $U(\h)$ is noetherian if and only if $U(\g)$ is noetherian.
\end{prop}

The key technique in the proof is to construct a filtration of $U(\g)$ whose associated graded ring is noetherian if $U(\h)$ is noetherian\footnote{The idea of this proof is due to Pavel Etingof, in personal correspondence with Susan Sierra and Chelsea Walton. We thank Prof. Etingof for allowing us to include it.}. This filtration will be constructed with the help of the following useful ideal of $\h$.

\begin{lem}\label{lem:k}
    Let $\h \subseteq \g$ be Lie algebras such that $\h$ has finite codimension in $\g$. Then $$\mf{k} = \{k \in \h \mid [k,x] \in \h \ \text{for all} \ x \in \g\}$$ is an ideal of $\h$ of finite codimension.
\end{lem}
\begin{proof}
    Consider the linear map
    \begin{align*}
        \varphi \colon \h &\to \mf{gl}(\g/\h) \\
        h &\mapsto \varphi_h
    \end{align*}
    induced by the adjoint action of $\h$ on $\g$. In other words, $$\varphi_h(x + \h) = [h,x] + \h$$
    for all $x \in \g$. We claim that $\varphi$ is a well-defined Lie algebra homomorphism. Consider the adjoint map
    \begin{align*}
        \ad \colon \h &\to \mf{gl}(\g) \\
        h &\mapsto \ad_h.
    \end{align*}
    Then $\varphi_h(x + \h) = \ad_h(x) + \h$ for all $h \in \h$. As $\ad_h$ preserves $\h$ for $h \in \h$, it follows that $\varphi$ is well-defined, as claimed.
    
    Note that $\mf{k} = \ker(\varphi)$, so $\mf{k}$ is an ideal of $\h$. Since $\mf{gl}(\g/\h)$ is finite-dimensional, it follows that $\h/\mf{k} \cong \varphi(\h) \subseteq \mf{gl}(\g/\h)$ is finite-dimensional.
\end{proof}

We now prove Proposition \ref{prop:finite codimension}.

\begin{proof}[Proof of Proposition \ref{prop:finite codimension}]
    By \cite[Lemma 1.7]{SierraWalton}, it suffices to prove that if $U(\h)$ is noetherian so is $U(\g)$.
    
    Suppose $U(\h)$ is noetherian. Let $\mf{k} \subseteq \h$ be as in Lemma \ref{lem:k} and let $x_1 + \mf{k}, \ldots, x_n + \mf{k}$ be a basis for $\h/\mf{k}$, where $x_i \in \h$. Similarly, let $y_1 + \h, \ldots, y_m + \h$ be a basis for $\g/\h$, where $y_i \in \g$. Define a filtration $\mc{F} = (F_i)_{i \in \NN}$ on $U(\g)$ as follows:
    $$F_i = \spn\{w x_{u_1} \ldots x_{u_r} y_{v_1} \ldots y_{v_s} \mid w \in U(\mf{k}), r + 2s \leq i\}.$$
    We claim that $F_i F_j \subseteq F_{i+j}$, so that $\mc{F}$ is indeed a filtration. First, we let $k \in \mf{k}$ and consider $x_\ell k$ for some $\ell$:
    $$x_\ell k = k x_\ell + [x_\ell, k]$$
    in $U(\mf{g})$. Note that $[x_\ell,k] \in \mf{k}$ since $\mf{k}$ is an ideal of $\mf{h}$. Therefore, $x_\ell k \in F_1$. Inductively, we see that $x_\ell w \in F_1$ for all $w \in U(\mf{k})$.
    
    Now consider $y_\ell k$:
    $$y_\ell k = k y_\ell + [y_\ell,k]$$
    in $U(\mf{g})$. Since $[y_\ell,k] \in \mf{h} \subseteq F_2$ and $k y_\ell \in F_2$, it follows that $y_\ell k \in F_2$. Inductively, we see that $y_\ell w \in F_2$ for all $w \in U(\mf{k})$.
    
    Finally, consider $y_\ell x_{\ell'}$ for some $\ell, \ell'$:
    $$y_\ell x_{\ell'} = x_{\ell'} y_\ell + [y_\ell, x_{\ell'}]$$
    in $U(\mf{g})$. It follows similarly to before that $y_\ell x_{\ell'} \in F_3$. By induction, $y_\ell x_{u_1} \ldots x_{u_r} \in F_{r+2}$.
    
    Let $z \in F_j$. Combining the above, we see that $x_\ell z \in F_{j+1}$ and $y_\ell z \in F_{j+2}$ for all $\ell$. Now, consider $z' \in F_i$ of the form
        $$z' = w x_{u_1} \ldots x_{u_r} y_{v_1} \ldots y_{v_s}, \quad r + 2s \leq i.$$
    Since $x_{u_r}z \in F_{j + 1}$ and $y_{v_s}z \in F_{j + 2}$, it now follows by an easy induction that $z'z \in F_{i + j}$. Therefore, $F_i F_j \subseteq F_{i+j}$, as claimed.
    
    Note that $F_0 = U(\mf{k})$, $x_i \in F_1 \setminus F_0$ and $y_i \in F_2 \setminus F_1$, so that $\h \subseteq F_1$ and $\g \subseteq F_2$. Consider the associated graded ring
    $$\gr_\mc{F} U(\g) = \bigoplus_{i \in \NN} F_i/F_{i-1}.$$
    Let $\overline{x_i} = x_i + F_0 \in \gr_\mc{F} U(\g)$ and $\overline{y_i} = y_i + F_1 \in \gr_\mc{F} U(\g)$. By definition of the filtration, it follows that $U(\mf{k}) \subseteq \gr_\mc{F} U(\g)$, and $\gr_\mc{F} U(\g)$ is generated by $U(\mf{k})$ and the $\overline{x_i}$ and $\overline{y_j}$ as a $\kk$-algebra. We have
    $$\overline{x_i} \ \overline{x_j} - \overline{x_j} \ \overline{x_i} = [x_i,x_j] + F_1.$$
    But $[x_i,x_j] \in \h \subseteq F_1$, so $[x_i,x_j] + F_1 = 0$ and therefore $\overline{x_i}$ and $\overline{x_j}$ commute in $\gr_\mc{F} U(\g)$ for all $i,j$. Similarly, we have
    $$\overline{y_i} \ \overline{y_j} - \overline{y_j} \ \overline{y_i} = [y_i,y_j] + F_3 = 0,$$
    $$\overline{x_i} \ \overline{y_j} - \overline{y_j} \ \overline{x_i} = [x_i,y_j] + F_2 = 0,$$
    since $[y_i,y_j] \in \g \subseteq F_2 \subseteq F_3$ and $[x_i,y_j] \in \g \subseteq F_2$.
    
    Now let $k \in \mf{k} \setminus \{0\} \subseteq F_0$. We have
    $$k\overline{x_i} - \overline{x_i}k = [k,x_i] + F_0 = 0,$$
    $$k\overline{y_i} - \overline{y_i}k = [k,y_i] + F_1 = 0,$$
    since $[k,x_i] \in \mf{k} \subseteq F_0$ by Lemma \ref{lem:k} and $[k,y_i] \in \h \subseteq F_1$ by definition of $\mf{k}$.
    
    Hence, the elements $\overline{x_i}$ and $\overline{y_j}$ are central in $\gr_\mc{F} U(\g)$, so we see that
    $$\gr_\mc{F} U(\g) = U(\mf{k})[\overline{x_1}, \ldots, \overline{x_n}, \overline{y_1}, \ldots, \overline{y_m}].$$
    Since $U(\h)$ is noetherian, it follows by \cite[Lemma 1.7]{SierraWalton} that $U(\mf{k})$ is noetherian. By Hilbert's basis theorem, $\gr_\mc{F} U(\g)$ is noetherian. Therefore, $U(\g)$ is noetherian since it has a positive filtration whose associated graded ring is noetherian.
\end{proof}

\section{Krichever-Novikov algebras}\label{sec:affine curves}

In this section, we focus on universal enveloping algebras of Krichever-Novikov algebras, which we define below. In order to define these, we first need to recall the notion of a derivation of an associative algebra.

\begin{dfn}
    Let $A$ be a $\kk$-algebra and let $C$ be an affine curve.
    \begin{enumerate}
        \item A \emph{$\kk$-derivation} of $A$ is a linear map $v \colon A \to A$ which satisfies the Leibniz rule
        $$v(ab) = av(b) + v(a)b$$
        for all $a, b \in A$. We denote the Lie algebra of $\kk$-derivations of $A$ by $\Der(A)$.
        \item A \emph{vector field on $C$} is a derivation of the coordinate ring $\kk[C]$.
        \item The \emph{Krichever-Novikov algebra on $C$}, denoted $\LL(C)$, is the Lie algebra of vector fields on $C$. In other words, $\LL(C) = \Der(\kk[C])$ with Lie bracket given by the commutator of derivations.
    \end{enumerate}
\end{dfn}

\begin{rem}
    The standard definition of a Krichever-Novikov algebra assumes that the underlying curve is nonsingular (see, for example, \cite{Schlichenmaier}). However, in this paper we also consider Lie algebras of vector fields on singular curves and also refer to them as Krichever-Novikov algebras.
\end{rem}

The goal of this section is to prove that universal enveloping algebras of Krichever-Novikov algebras are not noetherian, providing a new class of infinite-dimensional Lie algebras with non-noetherian enveloping algebras.

\begin{thm}\label{thm:main}
    Let $C$ be an affine curve. Then $U(\LL(C))$ is not noetherian.
\end{thm}

We establish notation for the Krichever-Novikov algebra on the affine line, which will be of central importance throughout the paper.

\begin{ntt}
    Let $W_{\geq -1} = \LL(\Aa^1) = \Der(\kk[t])$. Then
    $$W_{\geq -1} = \kk[t] \del,$$
    where $\del$ denotes differentiation by $t$. The Lie bracket in $W_{\geq -1}$ is given by
    $$[f\del,g\del] = (fg' - f'g)\del.$$
    For $f \in \kk[t]$, we let $W_{\geq -1}(f) = \kk[t] f \del$. It is clear that $W_{\geq -1}(f)$ is a Lie subalgebra of $W_{\geq -1}$. We write $W_{\geq n} = W_{\geq -1}(t^{n + 1})$ for $n \in \NN$.
\end{ntt}

\begin{rem}
    The notation $W_{\geq -1}$ comes from the fact that $W_{\geq -1}$ is a Lie subalgebra of the \emph{Witt algebra} $W = \Der(\kk[t,t^{-1}]) = \kk[t,t^{-1}]\del$. Letting $e_n = t^{n+1} \del \in W$, we see that $\{e_n \mid n \in \ZZ\}$ is a basis for $W$ and
    $$[e_n, e_m] = (m-n) e_{n+m}$$
    for all $n,m \in \ZZ$. Then $W_{\geq -1}$ is the Lie subalgebra of $W$ spanned by $\{e_n \mid n \geq -1\}$, and $W_{\geq n}$ is spanned by $\{e_m \mid m \geq n\}$. The \emph{positive Witt algebra} $W_+ = W_{\geq 1}$ is another noteworthy subalgebra of $W$.
\end{rem}

The proof of Theorem \ref{thm:main} will go as follows: if $C$ is nonsingular, we will show that there is an injective Lie algebra homomorphism 
$$\W(f) \hookrightarrow \LL(C),$$
for some $f \in \kk[t]$. We also show that $U(\W(f))$ is not noetherian, which is enough to conclude that $U(\LL(C))$ is not noetherian.

If $C$ is singular, we will consider the normalisation $\widetilde{C}$ of $C$, and construct a Lie subalgebra $\LL'$ of both $\LL(\widetilde{C})$ and $\LL(C)$ which has finite codimension in both. The non-noetherianity of $U(\LL')$ will follow from the non-noetherianity of $U(\LL(\widetilde{C}))$ by Proposition \ref{prop:finite codimension}, allowing us to conclude that $U(\LL(C))$ is not noetherian.

The following result is the main theorem of \cite{SierraWalton}.

\begin{thm}[{\cite[Theorem 0.5]{SierraWalton}}]\label{thm:Sue}
    The universal enveloping algebra of the positive Witt algebra $U(W_+)$ is not noetherian.
\end{thm}

It follows from Proposition \ref{prop:finite codimension} and Sierra and Walton's work that $W_{\geq -1}$ and its subalgebras of the form $W_{\geq -1}(f)$ have non-noetherian universal enveloping algebras.

\begin{cor}\label{cor:finite codimension Witt}
    Let $f \in \kk[t] \setminus \{0\}$. Then $U(\W(f))$ is not noetherian.
\end{cor}
\begin{proof}
    Since $W_+ \subseteq W_{\geq -1}$, it follows that $U(W_{\geq -1})$ is not noetherian by \cite[Lemma 1.7]{SierraWalton} and Theorem \ref{thm:Sue}. Note that $\dim_\kk (W_{\geq -1}/\W(f)) = \deg(f) < \infty$, so $U(\W(f))$ is not noetherian by Proposition \ref{prop:finite codimension}.
\end{proof}

In order to construct the homomorphism $\W(f) \hookrightarrow \LL(C)$ for the proof of Theorem \ref{thm:main}, some technical results about derivations are required. We will need to consider dominant morphisms between affine curves, and localisations of coordinate rings. We recall these notions below.

\begin{dfn}
    Let $X$ and $Y$ be affine curves. A morphism $\pi \colon X \to Y$ is \emph{dominant} if $\pi(X)$ is dense in $Y$ (in the Zariski topology).
\end{dfn}

\begin{rem}
    Let $X,Y$ be affine curves, and let $A = \kk[X], B = \kk[Y]$. If $\pi \colon X \to Y$ is a dominant morphism, then the homomorphism
    \begin{align*}
        \pi^* \colon B &\to A \\
        f &\mapsto f \circ \pi
    \end{align*}
    is injective, so we make the identification $B = \pi^*(B) \subseteq A$.
\end{rem}

\begin{dfn}
    Let $\pi \colon X \to Y$ be a dominant morphism of nonsingular affine curves. Let $A = \kk[X]$ and $B = \kk[Y] \subseteq A$. Let $K = Q(A)$ and $L = Q(B) \subseteq K$ be the fields of fractions of $A$ and $B$, respectively. Let $P \in X$ and $Q = \pi(P)$, and let
    $$A_P = \{\frac{f}{g} \in K \mid f,g \in A, g(P) \neq 0\},$$
    $$B_Q = \{\frac{f}{g} \in L \mid f,g \in B, g(Q) \neq 0\} \subseteq A_P.$$ The rings $A_P$ and $B_Q$ are discrete valuation rings, so let $s$ and $t$ be uniformising parameters of $A_P$ and $B_Q$, respectively. Let $\nu_P$ be the discrete valuation on $K$ associated to $A_P$ (i.e. $\nu_P(as^k) = k$ for $a \in A_P^*$ and $k \in \ZZ$). The \emph{ramification index} $e_P$ of $\pi$ at the point $P$ is defined as
    $$e_P = \nu_P(t) \geq 1$$
    where $t$ is viewed as an element of $A_P$. If $e_P > 1$, we say that $\pi$ is \emph{ramified} at $P$ and that $Q$ is a \emph{branch point} of $\pi$. If $e_P = 1$, we say that $\pi$ is \emph{unramified} at $P$.
\end{dfn}

The following proposition shows that the branch point locus of a dominant morphism is a finite set, a fact that we will use implicitly from now on.

\begin{prop}[{\cite[Proposition IV.2.2(a)]{Hartshorne}}]
    Let $\pi \colon X \to Y$ be a dominant morphism of nonsingular affine curves. Then $\pi$ is ramified at only finitely many points.
\end{prop}

The next result is a technical lemma about restricting derivations to the local ring $A_P$.

\begin{lem}\label{lem:derivation uniformising parameter}
    Let $X$ be a nonsingular affine curve, let $A = \kk[X]$, and let $K = Q(A)$. Let $P \in X$, and let $s \in A_P$ be a uniformising parameter. Then a derivation $v \in \Der(K)$ restricts to a derivation of $A_P$ if and only if $v(s) \in A_P$.
\end{lem}
\begin{proof}
    Suppose $v \in \Der(K)$ such that $v(s) \in A_P$. Let $\mf{m} = (s)$ be the maximal ideal of $A_P$, and consider the completion
    $$\widehat{A_P} = \varprojlim_{k \in \NN} A_P/\mf{m}^k \cong \kk[[s]].$$
    Then $v$ extends uniquely to a derivation $\hat{v}$ of $Q(\widehat{A_P}) \cong \kk((s))$. We have
    $$\hat{v}(s) = v(s) \in A_P \subseteq \widehat{A_P}.$$
    It follows that $\hat{v}(\widehat{A_P}) \subseteq \widehat{A_P}$, since $\widehat{A_P} \cong \kk[[s]]$ and $\hat{v}$ is $\kk$-linear. Hence, we have
    $$v(A_P) = \hat{v}(A_P) = \hat{v}(\widehat{A_P} \cap K) \subseteq \widehat{A_P} \cap K = A_P,$$ and therefore $v$ restricts to a derivation of $A_P$.
\end{proof}

\begin{dfn}
    Let $X$ be an affine curve. We say that a vector field $v \in \LL(X)$ \emph{vanishes at} $x \in X$ if $v(f)(x) = 0$ for all $f \in \kk[X]$. If $S$ is a finite subset of $X$, we denote the Lie subalgebra of $\LL(X)$ consisting of vector fields which vanish on $S$ by $\LL_S(X)$.
\end{dfn}

The following proposition shows that if we have a dominant morphism $\pi \colon X \to Y$ of nonsingular curves, a vector field on $Y$ induces a vector field on $X$ provided it vanishes on the branch point locus of $\pi$. This gives an injective homomorphism from a subalgebra of $\LL(Y)$ to $\LL(X)$, which will be one of the main ingredients in the proof of Theorem \ref{thm:main}.

\begin{prop}\label{prop:extending derivations}
    Let $X$ and $Y$ be nonsingular affine curves, and let $\pi \colon X \to Y$ be a dominant morphism. Let $S \subseteq Y$ be the branch point locus of $\pi$. Then any vector field $v \in \LL_S(Y)$ extends uniquely to a vector field $\widetilde{v} \in \LL(X)$ and the map
    \begin{align*}
        \varphi \colon \LL_S(Y) &\to \LL(X) \\
        v &\mapsto \widetilde{v}
    \end{align*}
    defines an injective Lie algebra homomorphism.
\end{prop}
\begin{proof}
    Let $A = \kk[X]$ and $B = \kk[Y] \subseteq A$. Let $v \in \LL(Y) = \Der(B)$ be a vector field which vanishes on $S$. Since $\pi$ is dominant, we identify $B$ with $\pi^*(B) \subseteq A$. Let $K = Q(A)$ and $L = Q(B) \subseteq K$ be the fields of fractions of $A$ and $B$, respectively. Then $v$ extends uniquely to some $v_L \in \Der(L)$. Since $\pi$ is dominant, it follows by \cite[Proposition II.6.8]{Hartshorne} that $\dim_L K < \infty$, so $v_L$ extends uniquely to some $v_K \in \Der(K)$. Let $\widetilde{v} \colon A \to K$ be the restriction of $v_K$ to $A$. We claim that $\widetilde{v}(A) \subseteq A$, so that $\widetilde{v} \in \Der(A) = \LL(X)$.
    
    For $P \in X$, let $A_P = \{\frac{f}{g} \in K \mid f,g \in A, g(P) \neq 0\}$, and let $s_P \in A_P$ be a uniformising parameter of $A_P$. By Lemma \ref{lem:derivation uniformising parameter}, $v_K(A_P) \subseteq A_P$ if and only if $v_K(s_P) \in A_P$. Since
    $$A = \bigcap_{P \in X} A_P,$$
    it then follows that that $\widetilde{v}(A) \subseteq A$ if and only if $v_K(s_P) \in A_P$ for all $P \in X$. Hence, it suffices to show that $v_K(s_P) \in A_P$ for an arbitrary $P \in X$.
    
    Fix $P \in X$ and let $Q = \pi(P) \in Y$. Let $B_Q = \{\frac{f}{g} \in L \mid f,g \in B, g(Q) \neq 0\}$, and let $t \in B_Q$ be a uniformising parameter of $B_Q$. For simplicity of notation, let $s = s_P$. Since $t \in B_Q \subseteq A_P$ is a uniformising parameter of $B_Q$, it follows that $t = a s^e$ for some $a \in A_P^*$, where $e = e_P$ is the ramification index of $\pi$ at $P$. Furthermore, $v(t) \in B_Q$ so $v(t) = b t^f$ for some $b \in B_Q^*$ and some $f \geq 0$. Note that $f$ is the order of vanishing of $v(t)$ at the point $Q$, so in particular $f \geq 1$ if $v$ vanishes at $Q$.
    
    As in the proof of Lemma \ref{lem:derivation uniformising parameter}, consider the completion $\widehat{A_P} \cong \kk[[s]]$ and the unique extension of $v$ to a derivation $\hat{v}$ of $Q(\widehat{A_P}) \cong \kk((s))$. Similarly, the completion $\widehat{B_Q} \subseteq \widehat{A_P}$ is isomorphic to $\kk[[t]]$. Now, we have
    \begin{equation}\label{eq:t}
        t = s^e \sum_{i \geq 0} \alpha_i s^i \in \widehat{A_P}
    \end{equation}
    \begin{equation}\label{eq:v(t)}
        v(t) = t^f \sum_{i \geq 0} \beta_i t^i \in \widehat{B_Q}
    \end{equation}
    for some $\alpha_i, \beta_i \in \kk$ with $\alpha_0, \beta_0 \neq 0$. From \eqref{eq:t}, we get
    \beq\label{eq:v(s)}
        v(t) = \hat{v}(s^e \sum_{i \geq 0} \alpha_i s^i) = v_K(s)s^{e-1} \sum_{i \geq 0} \gamma_i s^i,
    \eeq
    where $\gamma_i = (e + i)\alpha_i \in \kk$. Since $\gamma_0 = e \alpha_0 \neq 0$, the element $H = \sum_{i \geq 0} \gamma_i s^i \in \widehat{A_P}$ is invertible in $\widehat{A_P} \cong \kk[[s]]$. Let $F = H^{-1} \in \widehat{A_P}$.
    
    We can also combine \eqref{eq:t} and \eqref{eq:v(t)} to get
    \beq\label{eq:s^ef}
        v(t) = t^f \sum_{i \geq 0} \beta_i t^i = s^{ef} \sum_{i \geq 0} \mu_i s^i,
    \eeq
    for some $\mu_k \in \kk$.
    
    Let $G = \sum_{i \geq 0} \mu_i s^i \in \widehat{A_P}$. Equating \eqref{eq:v(s)} and \eqref{eq:s^ef}, we get $v_K(s)s^{e-1} H = s^{ef} G$, and therefore
    $$v_K(s) = s^{e(f - 1) + 1} FG.$$
    If $e = 1$, then $e(f-1)+1 \geq 0$ since $f \geq 0$. If $e \geq 2$, then the point $P$ is ramified, so $v$ vanishes at $Q$ by assumption. Therefore, $f \geq 1$ and $e(f-1) + 1 \geq 1$. In either case, we see that $e(f-1) + 1 \geq 0$. Hence, $v_K(s) = s^{e(f - 1) + 1} FG \in \widehat{A_P}$.
    
    But we also know that $v_K(s) \in K$, so that
    $$v_K(s) \in \widehat{A_P} \cap K = A_P.$$
    We conclude that $\widetilde{v} \in \Der(A)$, as claimed.
    
    It follows that the map $\varphi \colon \LL_S(Y) \to \LL(X)$ defined in the statement of the proposition is a Lie algebra homomorphism by uniqueness of the extension: if $v_1, v_2 \in \LL_S(Y)$, then the vector field
    $$[\widetilde{v_1}, \widetilde{v_2}] = \widetilde{v_1} \circ \widetilde{v_2} - \widetilde{v_2} \circ \widetilde{v_1} \in \LL(X)$$
    restricts to
    $$\restr{[\widetilde{v_1}, \widetilde{v_2}]}{B} = v_1 \circ v_2 - v_2 \circ v_1 = [v_1, v_2] \in \LL_S(Y).$$
    By uniqueness of the extension, we conclude that
    $$[\varphi(v_1),\varphi(v_2)] = [\widetilde{v_1},\widetilde{v_2}] = \varphi([v_1,v_2]).$$
    Furthermore, $\varphi$ is clearly injective since $\widetilde{v}$ is an extension of $v$.
\end{proof}

We now have all the necessary tools to prove Theorem \ref{thm:main}.

\begin{proof}[Proof of Theorem \ref{thm:main}]
    First we consider the case where $C$ is nonsingular. There is an injective homomorphism $\kk[t] \hookrightarrow \kk[C]$, which induces a dominant morphism $\pi \colon C \to \Aa^1$. Let $S = \{x_1, \ldots, x_n\}$ be the branch point locus of $\pi$, and let $f = (t - x_1) \ldots (t - x_n) \in \kk[t]$. Consider
    $$\LL_S(\Aa^1) = W_{\geq -1}(f),$$
    and let
    \begin{align*}
        \varphi \colon \LL_S(\Aa^1) &\hookrightarrow \LL(C) \\
        v &\mapsto \widetilde{v}
    \end{align*}
    be the injective Lie algebra homomorphism from Proposition \ref{prop:extending derivations}. By Corollary \ref{cor:finite codimension Witt}, we know that $U(\LL_S(\Aa^1)) = U(\W(f))$ is not noetherian, and therefore $U(\LL(C))$ is not noetherian.
    
    Now suppose $C$ is singular, and let $\widetilde{C}$ be its normalisation. Set the following notation:
    \begin{align*}
        A = \kk[C], \quad &\widetilde{A} = \kk[\widetilde{C}], \\
        \mc{L} = \LL(C), \quad &\widetilde{\mc{L}} = \LL(\widetilde{C}).
    \end{align*}
    We will view $A$ as a subring of $\widetilde{A}$. Let
    $$I = \Ann_A(\widetilde{A}/A) = \{a \in A \mid a\widetilde{A} \subseteq A\}.$$
    Then $I$ is an ideal of both $A$ and $\widetilde{A}$. Letting $v \in \widetilde{\LL}$ and $a \in I$, we have that
    $$av(A) \subseteq A.$$
    It follows that $av$ is a derivation of $A$, that is $av \in \LL$. Therefore, $I\widetilde{\LL} \subseteq \LL$.
    
    We can check that
    $$[a_1v_1,a_2v_2] = a_1a_2[v_1,v_2] + a_1v_1(a_2)v_2 - a_2v_2(a_1)v_1 \in I\widetilde{\LL},$$
    where $a_1,a_2 \in I$ and $v_1,v_2 \in \widetilde{\LL}$. Hence, $I\widetilde{\LL}$ is a Lie subalgebra of both $\LL$ and $\widetilde{\LL}$.
    
    Since $\LL$ and $\widetilde{\LL}$ are rank 1 torsion-free modules over $A$ and $\widetilde{A}$, respectively, it follows that $I\widetilde{\LL}$ has finite codimension in both $\LL$ and $\widetilde{\LL}$. We have already shown that $U(\widetilde{\LL})$ is not noetherian, so it follows that $U(I\widetilde{\LL})$ is not noetherian by Proposition \ref{prop:finite codimension}. Therefore, $U(\LL)$ is not noetherian.
\end{proof}

We finish this section by briefly mentioning some other related Lie algebras.

\begin{dfn}\label{def:KN type}
    Let $C$ be an affine curve and let $\g$ be a finite-dimensional Lie algebra. We define the following Lie algebras:
    \begin{enumerate}
        \item The \emph{Lie algebra of differential operators of degree $\leq 1$ on $C$} is defined as the semidirect sum $\mc{D}^1(C) = \kk[C] \rtimes \LL(C)$. In other words, $\mc{D}^1(C) = \kk[C] \oplus \LL(C)$ as a vector space, with Lie bracket given by
        $$[(f_1,v_1),(f_2,v_2)] = (v_1(f_2) - v_2(f_1),[v_1,v_2]),$$
        where $f_1,f_2 \in \kk[C]$ and $v_1,v_2 \in \LL(C)$.
        \item The \emph{current algebra} $\overline{\g}(C) = \g \otimes_{\kk} \kk[C]$ has Lie bracket
        $$[x \otimes f, y \otimes g] = [x,y] \otimes fg,$$
        where $x,y \in \g$ and $f,g \in \kk[C]$.
    \end{enumerate}
    These Lie algebras, as well as $\LL(C)$, are called \emph{Krichever-Novikov type algebras} (cf. \cite[Definition 2.21]{Schlichenmaier}).
\end{dfn}

An easy consequence of Theorem \ref{thm:main} is that the enveloping algebra of a Krichever-Novikov type algebra is not noetherian.

\begin{cor}
    Krichever-Novikov type algebras have non-noetherian universal enveloping algebras.
\end{cor}
\begin{proof}
    Let $C$ be an affine curve and let $\g$ be a finite-dimensional Lie algebra. Consider $\overline{\g}(C)$ and $\mc{D}^1(C)$.
    
    To see that $U(\overline{\g}(C))$ is not noetherian, we simply note that $\overline{\g}(C)$ has an infinite-dimensional abelian Lie subalgebra, which clearly has a non-noetherian enveloping algebra. For example, we can take the subalgebra spanned by $\{x \otimes f^n \mid n \in \NN\}$, where $x \in \g \setminus \{0\}$ and $f \in \kk[C] \setminus \{0\}$.
    
    The non-noetherianity of $U(\mc{D}^1(C))$ follows directly from Theorem \ref{thm:main}, since $\mc{D}^1(C)$ contains $\LL(C)$ as a Lie subalgebra.
\end{proof}

\section{Subalgebras of $\W$ of infinite codimension}\label{sec:subalgebras}

We would like to prove that any infinite-dimensional subalgebra of a Krichever-Novikov algebra has a non-noetherian universal enveloping algebra. In order to achieve this, we first focus on subalgebras of $W_{\geq -1}$, with the eventual aim of extending the results to arbitrary Krichever-Novikov algebras using similar methods to those in Section \ref{sec:affine curves}.

A special case of Conjecture 0.1 of \cite{SierraWalton} is:

\begin{conj}\label{conj:subalgebras of W}
    Let $\mf{g}$ be an infinite-dimensional Lie subalgebra of $W_{\geq -1}$. Then $U(\mf{g})$ is not noetherian.
\end{conj}

The main difficulty in proving this conjecture is that we do not know what a general subalgebra of $\W$ looks like. Therefore, we first focus on graded subalgebras of $W_{\geq -1}$ and show that Conjecture \ref{conj:subalgebras of W} holds for these.

\subsection{Veronese and graded subalgebras}\label{subsec:Veronese}

Since many of the details are straightforward, we omit some of the proofs in this subsection.

Note that $\W$ is a graded Lie algebra. Under the standard grading of $\W$, the element $e_n$ is homogeneous of degree $n$. However, this grading is not unique: fix $x \in \kk^*$, and define $e_n(x) = (t - x)^{n + 1}\del$ for all $n \geq -1$. It is easy to see that
$$[e_n(x), e_m(x)] = (m - n) e_{n + m}(x).$$
Hence, we could also have chosen a different grading of $\W$, where $e_n(x)$ is a homogeneous element of degree $n$. We call this grading the \emph{grading of $\W$ based at $x$}. Note that $e_n$ is no longer homogeneous under this grading. We can therefore consider Veronese subalgebras with respect to the different choices of grading of $\W$.

\begin{dfn}\label{def:Veronese}
    Let $d \geq 2$ be an integer and let $x \in \kk$. The \emph{$d$-Veronese subalgebra of $W_{\geq -1}$ based at $x$}, denoted $\Ver_d(x)$, is the Lie subalgebra of $W_{\geq -1}$ with basis
    $$\{e_0(x), e_d(x), e_{2d}(x), \ldots\}.$$
    For brevity, when $x = 0$ we simply write $\Ver_d$ instead of $\Ver_d(0)$.
\end{dfn}

\begin{rem}
    Veronese subalgebras are the simplest infinite-dimensional subalgebras of $W_{\geq -1}$ of infinite codimension. They are therefore the first examples of subalgebras of $W_{\geq -1}$ for which Proposition \ref{prop:finite codimension} does not apply.
\end{rem}

It is rather straightforward to prove that these subalgebras have non-noetherian universal enveloping algebras; we merely have to notice that they are isomorphic to $W_{\geq 0}$.

\begin{lem}\label{lem:enveloping algebra of Veronese}
    Let $d \geq 2$ be an integer and let $x \in \kk$. Then $U(\Ver_d(x))$ is not noetherian.
\end{lem}
\begin{proof}
    Let $f_n = \frac{1}{d} e_{nd}(x) \in \Ver_d(x)$. Then $\{f_n \mid n \in \NN\}$ is a basis for $\Ver_d(x)$, and
    $$[f_n,f_m] = \frac{1}{d^2} [e_{nd}(x),e_{md}(x)] = \frac{1}{d^2} (md - nd) e_{(n+m)d}(x) = (m - n)f_{n+m}.$$
    Therefore, $\Ver_d(x) \cong W_{\geq 0}$, where the isomorphism maps $f_n \mapsto e_n$. By Corollary \ref{cor:finite codimension Witt}, $U(W_{\geq 0})$ is not noetherian, so $U(\Ver_d(x))$ is not noetherian.
\end{proof}

Our next goal is to prove that general infinite-dimensional graded subalgebras of $\W$ have non-noetherian universal enveloping algebras. For simplicity, we only consider the standard grading of $\W$, but the same results will hold for the different choices of grading. This is because the map
\begin{align*}
    \W &\to \W^{(x)} \\
    e_n &\mapsto e_n(x)
\end{align*}
is a graded isomorphism of Lie algebras, where $\W^{(x)}$ is $\W$ with grading based at $x$.

Graded subalgebras of $W_{\geq -1}$ are some of the easiest to understand since they have a basis which is a subset of $\{e_n \mid n \geq -1\}$. We will mostly be concerned with leading terms of elements of $\W$, so we establish the following notation.

\begin{ntt}
    For $\lambda \in \kk$ and $n \in \NN$, when we write $a = \lambda e_n + \ldots$, we mean that there exist $\alpha_1, \ldots, \alpha_{n+1} \in \kk$ such that
    $$a = \lambda e_n + \alpha_1 e_{n-1} + \ldots + \alpha_n e_0 + \alpha_{n+1} e_{-1}.$$
    For non-homogeneous $a \in \W$, we write $\deg(a) = n$ to mean $a = \lambda e_n + \ldots$ for some $\lambda \in \kk^*$.
\end{ntt}

\begin{rem}
    Let $a,b \in W_{\geq -1}$ with $\deg(a) \neq \deg(b)$. Then $\deg([a,b]) = \deg(a) + \deg(b)$.
\end{rem}

The following result gives a way of generating elements of high degree in a subalgebra $\g \subseteq \W$.

\begin{lem}\label{lem:gcd}
    Let $\g$ be a subalgebra of $W_{\geq -1}$, let $a,b \in \g$ and let $n = \deg(a)$, $m = \deg(b)$. Suppose $n,m \geq 1$ and $n \neq m$, and let $d = \gcd(n,m)$. Then there exists $k \in \NN$ such that for all $\ell \geq k$, there is an element $c_\ell \in \g$ with $\deg(c_\ell) = \ell d$.
\end{lem}
\begin{proof}
    For all $i,j \geq 1$, by taking Lie brackets of $a$ and $b$ we see that there is an element $h_{ij} \in \g$ such that $\deg(h_{ij}) = in + jm$. For example, we can set $h_{ij} = \ad_b^{j-1}(\ad_a^i(b))$.
    
    By the Euclidean algorithm, there exist $u,v \in \NN$ such that
    $$un - vm = d.$$
    Let $r = \frac{n}{d}$, let $s = n + ((r-1)v + 1)m$, and let $k = \frac{s}{d}$. Setting $c_k = h_{1,(r-1)v + 1} \in \g$, we see that $\deg(c_k) = s = kd$.
    
    Let $\ell \in \NN$ such that $k \leq \ell < k + r$, and let $q = \ell - k$. We have
    \begin{align*}
        \ell d &= (k + q)d = s + qd = (n + ((r - 1)v + 1)m) + q(un - vm) \\
        &= (qu + 1)n + ((r - q - 1)v + 1)m
    \end{align*}
    with $qu + 1 \geq 1$ and $(r - q - 1)v + 1 \geq 1$, since $0 \leq q < r$. Setting
    $$c_\ell = h_{qu+1,(r-q-1)v + 1} \in \g,$$
    we see that $\deg(c_\ell) = \ell d$.
    
    Now consider $\ell \geq k + r$. There exists $N \in \NN$ such that $k \leq \ell - Nr < k + r$. We have already constructed $c_{\ell - Nr} \in \g$. Define $c_\ell = \ad_a^N(c_{\ell-Nr}) \in \g$. Then
    $$\deg(c_\ell) = (\ell - Nr)d + Nn = \ell d - Nrd + Nn = \ell d,$$
    since $rd = n$, so we are done.
\end{proof}

When applied to graded subalgebras, Lemma \ref{lem:gcd} yields the following corollary.

\begin{cor}\label{cor:subalgebra in Veronese}
    Let $\mf{g}$ be an infinite-dimensional graded Lie subalgebra of $W_{\geq -1}$. Then either $\mf{g}$ has finite codimension in $W_{\geq -1}$ or $\mf{g}$ has finite codimension in $\Ver_d$ for some $d \geq 2$. \hfill \qed
\end{cor}

Combining the results above, it follows that infinite-dimensional graded subalgebras of $W_{\geq -1}$ have non-noetherian enveloping algebras.

\begin{prop}\label{prop:graded not noetherian}
    Let $\mf{g}$ be an infinite-dimensional graded Lie subalgebra of $W_{\geq -1}$. Then $U(\mf{g})$ is not noetherian. \hfill \qed
\end{prop}

\subsection{The subalgebras $L(f,g)$}\label{subsec:L(f,g)}

We turn our attention to subalgebras of $W_{\geq -1}$ which are not necessarily graded. If $\g \subseteq \Ver_d(x)$ for some $x \in \kk$, where $d \geq 2$ is maximal, then consider the isomorphism
$$\varphi \colon \Ver_d(x) \xrightarrow{\sim} W_{\geq 0}$$
from Lemma \ref{lem:enveloping algebra of Veronese}. By maximality of $d$, it follows that $\varphi(\g)$ is not contained in any Veronese subalgebra.

Veronese subalgebras, of course, have infinite codimension in $\W$. It is not immediately obvious that there are any infinite-dimensional subalgebras of $\W$ of infinite codimension which do not have finite codimension in a Veronese subalgebra; the goal of this subsection is to provide examples of such. By the previous paragraph, without loss of generality we seek for an infinite-dimensional subalgebra of $\W$ of infinite codimension which is not contained in a Veronese subalgebra.

The following result gives a useful method to determine whether some element of $W_{\geq -1}$ is contained in a Veronese subalgebra.

\begin{lem}\label{lem:contained in Veronese?}
    Let $a = e_n + \alpha e_{n-1} + \ldots \in W_{\geq -1}$. If $a \in \Ver_d(x)$ for some $d \geq 2$ and some $x \in \kk$, then $\alpha = -(n + 1)x$. In particular, either $a \in \Ver_d(\frac{-\alpha}{n + 1})$ for some $d \geq 2$ or $a$ is not contained in any Veronese subalgebra.
\end{lem}
\begin{proof}
    Suppose $a \in \Ver_d(x)$ for some $d \geq 2$ and some $x \in \kk$. Then
    $$a = e_n(x) + \beta e_{n-d}(x) + \ldots,$$
    for some $\beta \in \kk$. Therefore, we have
    $$a = (t - x)^{n + 1}\del + \beta (t - x)^{n - d + 1}\del + \ldots = e_n - (n + 1)x e_{n-1} + \ldots \quad (\text{since } d \geq 2)$$
    and hence we see that $\alpha = -(n + 1)x$.
\end{proof}

We now define a new family of subalgebras of $\W$. Members of this family will provide examples of infinite-dimensional subalgebras of $\W$ of infinite codimension which are not contained in any Veronese subalgebra.

\begin{ntt}
    For $f,g \in \kk[t] \setminus \{0\}$, we let $L(f,g)$ be the subspace of $W_{\geq -1}$ spanned by $\{f^n g\del \mid n \in \NN\}$. In other words, $L(f,g) = \kk[f]g\del$.
\end{ntt}

Note that if $\deg(f) > 1$, then $L(f,g)$ has infinite codimension in $W_{\geq -1}$. The next lemma characterises when $L(f,g)$ is a Lie algebra.

\begin{lem}\label{lem:infinite codimension subalgebra}
    Let $f,g \in \kk[t]$. Then $L(f,g)$ is a Lie subalgebra of $W_{\geq -1}$ if and only if $f'g \in \kk[f]$ (i.e. $f'g = h(f)$ for some $h \in \kk[t]$). Furthermore, if $f'g = h(f)$, then $L(f,g) \cong \W(h)$.
\end{lem}
\begin{proof}
    Note that
    $$[f^ng\del,f^mg\del] = (m - n)f^{n+m-1}f'g^2\del.$$
    It is now clear that $L(f,g) = \kk[f]g\del$ is closed under the Lie bracket of $\W$ if and only if $f'g \in \kk[f]$.
    
    For the final sentence, suppose $f'g = h(f)$. We can easily check that the map
    \begin{align*}
        L(f,g) &\to \W(h) \\
        f^ng\del &\mapsto t^nh\del
    \end{align*}
    is an isomorphism of Lie algebras.
\end{proof}

The following proposition shows that $f$ can be arbitrary, and that for a fixed $f \in \kk[t]$ there is a maximal Lie algebra of the form $L(f,g)$.

\begin{prop}\label{prop:L(f)}
    Let $f \in \kk[t]$ be a non-constant polynomial. Then there exists a (unique up to scalar) polynomial $g_f \in \kk[t] \setminus \{0\}$ such that
    \begin{enumerate}
        \item $f'g_f \in \kk[f]$,
        \item If $h \in \kk[t]$ such that $f'h \in \kk[f]$ then $\deg(h) \geq \deg(g_f)$.
    \end{enumerate}
    Furthermore, if $f'h \in \kk[f]$ then $h \in \kk[f]g_f$ and $L(f,h)$ has finite codimension in $L(f,g_f)$. We write $L(f) = L(f,g_f)$.
\end{prop}
\begin{proof}
    Let
    $$I(f) = \{P \in \kk[t] \mid f' \text{ divides } P(f)\}.$$
    It is immediate that $I(f)$ is an ideal of $\kk[t]$. We claim that $I(f) \neq 0$. Let $Q \in \kk[t]$ be a polynomial that vanishes at $\{f(\alpha) \mid \alpha \in \overline{\kk}, f'(\alpha) = 0\}$. Then $Q(f)$ vanishes at all the roots of $f'$. Therefore, if we take $n \in \NN$ large enough, we have that $P = Q^n \in I(f) \setminus \{0\}$.
    
    The ideal $I(f)$ is nonzero and principal, so there exists some $P_f \in \kk[t] \setminus \{0\}$ such that
    $$I(f) = (P_f).$$
    Let $g_f = P_f(f)/f'$.
    
    Suppose we have $h \in \kk[t]$ such that $f'h \in \kk[f]$. Write
    $$f'h = P(f),$$
    where $P \in \kk[t]$. Then $P \in I(f) = (P_f)$, so $P_f$ divides $P$. Therefore, there exists some $Q \in \kk[t]$ such that $P = P_f Q$. It follows that
    $$h = P(f)/f' = P_f(f)Q(f)/f' = Q(f)g_f.$$
    Hence,
    $$L(f,h) = \kk[f]h\del = \kk[f]Q(f)g_f\del \subseteq \kk[f]g_f\del = L(f).$$
    It is easy to see that $\dim L(f,h)/L(f) = \deg Q < \infty$.
\end{proof}

\begin{rem}\label{rem:degree of L(f,g)}
    Let $f,g,h \in \kk[t]$ such that $f'g = h(f)$. We have
    $$\deg(f) - 1 +\deg(g) = \deg(h)\deg(f),$$
    and therefore
    $$\deg(g) = (\deg(h) - 1)\deg(f) + 1.$$
    Consider $p(f)g\del \in L(f,g)$, where $p \in \kk[t]$. Then
    \begin{align*}
        \deg(p(f)g\del) &= \deg(p(f)g) - 1 = \deg(p)\deg(f) + \deg(g) - 1 \\
        &= (\deg(p) + \deg(h) - 1)\deg(f).
    \end{align*}
    Therefore, the degrees of elements of $L(f,g)$ are multiples of $\deg(f)$.
\end{rem}

Since we already know that $U(\W(h))$ is not noetherian, the following result follows immediately from Lemma \ref{lem:infinite codimension subalgebra}.

\begin{prop}\label{prop:U(L(f,g) not noetherian}
    $U(L(f,g))$ is not noetherian for any $f,g \in \kk[t] \setminus \{0\}$ such that $f'g \in \kk[f]$ and $\deg(f) \geq 1$. \hfill \qed
\end{prop}

The example below shows that we can construct subalgebras of infinite codimension in $\W$ which are not contained in any Veronese subalgebra.

\begin{ex}
    If we let $f = t^3 + 3t$ then we claim that $g_f = (t^2 + 1)(t^2 + 4)$. Indeed, $f'g_f = 3(f^2 + 4) \in \kk[f]$, and it is easy to see that $f'h \not\in \kk[f]$ for all $h \in \kk[t]$ of degree 1, so $g_f$ has minimal degree.
    
    Note that $L(f)$ is an infinite-dimensional subalgebra of $W_{\geq -1}$ of infinite codimension. Furthermore, we can easily use Lemma \ref{lem:contained in Veronese?} to show that $L(f)$ is not contained in any Veronese subalgebra.
\end{ex}

\section{Working toward a classification of subalgebras of $\W$}\label{subsec:contained in L(f)}

In this section, we attempt a classification of infinite-dimensional subalgebras of $\W$. Subalgebras of finite codimension in $\W$ have already been classified in \cite{PetukhovSierra}: if $\g$ has finite codimension in $\W$ then there exist $f \in \kk[t] \setminus \{0\}, n \in \NN$ such that
$$\W(f^n) \subseteq \g \subseteq \W(f).$$
Therefore, it remains to classify subalgebras of $\W$ of infinite codimension. We believe that the subalgebras $L(f,g)$ from Subsection \ref{subsec:L(f,g)} are infinite-codimensional analogues of the subalgebras $\W(f)$, in the following sense:
\begin{conj}\label{conj:classification}
    If $\g$ is an infinite-dimensional subalgebra of $\W$ then there exist $f,g \in \kk[t]$ such that $f'g \in \kk[f]$ and
    $$L(f,g) \subseteq \g \subseteq L(f).$$
    In particular, $\g$ has finite codimension in $L(f)$.
\end{conj}

Note that if Conjecture \ref{conj:classification} holds, then so does Conjecture \ref{conj:subalgebras of W}.

In fact, in order to prove Conjecture \ref{conj:classification}, it suffices to prove that if $\g$ is an infinite-dimensional subalgebra of $\W$ then there exists $f \in \kk[t]$ such that $\g$ has finite codimension in $L(f)$. This can be seen by noting that if $\g$ has finite codimension in $L(f)$ then $\g$ is isomorphic to a subalgebra of finite codimension in $\W$, by Lemma \ref{lem:infinite codimension subalgebra}. Thus, we can apply Petukhov and Sierra's result to deduce that $\g$ contains a subalgebra isomorphic to $\W(h)$, for some $h \in \kk[t] \setminus \{0\}$. This subalgebra is of the form $L(f,g)$, for some $g \in \kk[t] \setminus \{0\}$.

In this section, we show that the subalgebra $\g$ must be contained in some $L(f)$ in a canonical way, but we have not been able to show that $\g$ has finite codimension in $L(f)$. In fact, we do not know of any infinite-dimensional subalgebra of $W_{\geq -1}$ which does not have finite codimension in some $L(f)$.

Let $\g$ be a subalgebra of $\W$ and suppose $\g$ has finite codimension in $L(f)$ for some $f \in \kk[t]$. If $\g = L(f)$ then by considering ratios $\frac{w}{u}$, where $w\del, u\del \in \g$, we can generate any element of $\kk(f)$. This suggests that looking at such ratios should give us a way of extracting the polynomial $f$ from the subalgebra $\g$. To this end, we introduce the following invariants of subalgebras of $\W$:

\begin{dfn}
    If $\g$ is a subalgebra of $\W$ then we define the \emph{set of ratios} $R(\g)$ of $\g$ by
    $$R(\g) = \left\{\frac{w}{u} \in \kk(t) \mid w\del, u\del \in \g, u \neq 0\right\}.$$
    We define the \emph{field of ratios} $F(\g)$ of $\g$ as the subfield of $\kk(t)$ generated by $R(\g)$.
\end{dfn}

We start by considering the field of ratios: the following result shows that $F(\g) = \kk(f)$ for some $f \in \kk[t]$ and, furthermore, $\g \subseteq L(f)$.

\begin{thm}\label{thm:subfield generated by R(g)}
    Let $\g$ be an infinite-dimensional Lie subalgebra of $\W$ and let $K \subseteq \kk(t)$ be a field containing $R(\g)$. Then there exists some $f \in \kk[t] \setminus \{0\}$ such that:
    \begin{enumerate}
        \item The field $K$ is generated by $f$.
        \item The subalgebra $\g$ is contained in $L(f)$.
    \end{enumerate}
\end{thm}

We leave the proof of Theorem \ref{thm:subfield generated by R(g)} for later, and instead focus on the set of ratios $R(\g)$. Certainly, $R(\g)$ is closed under scalar multiplication and taking multiplicative inverses, but it is not clear if it satisfies any other properties. As suggested by Theorem \ref{thm:subfield generated by R(g)}, it will be of particular interest to consider situations where $R(\g)$ is a field. In fact, we have:

\begin{thm}\label{thm:S is a field?}
    Let $\g$ be an infinite-dimensional Lie subalgebra of $\W$. The following are equivalent:
    \begin{enumerate}
        \item[(1)] There exists a nonzero $f \in \kk[t]$ such that $\g$ has finite codimension in $L(f)$.
        \item[(2)] The set $R(\g)$ is a field.
        \item[(3)] There exists a nonzero $f \in \kk[t]$ such that $R(\g) = \kk(f)$.
    \end{enumerate}
\end{thm}
\begin{proof}
    The implication $(3) \Rightarrow (2)$ is clear. We will show that $(1) \Rightarrow (3)$ and $(2) \Rightarrow (1)$.
    
    $(1) \Rightarrow (3)$: Suppose there exists $f \in \kk[t]$ such that $\g$ has finite codimension in $L(f)$. Then $R(\g) \subseteq \kk(f)$. We will show that $R(\g) = \kk(f)$.
    
    Let $g = g_f$ in the notation of Proposition \ref{prop:L(f)}, so $L(f) = L(f,g)$, and let $h \in \kk[t]$ such that $f'g = h(f)$. Consider the isomorphism
    \begin{align*}
        \varphi \colon L(f) &\to \W(h) \\
        p(f)g\del &\mapsto ph\del
    \end{align*}
    from Lemma \ref{lem:infinite codimension subalgebra}, where $p \in \kk[t]$. Let $\mf{k} = \varphi(\g)$. Since $\g$ has finite codimension in $L(f)$ and $\W(h)$ has finite codimension in $\W$, it follows that $\mf{k}$ has finite codimension in $\W$. By an analogous result to \cite[Proposition 3.19]{PetukhovSierra} (which considers subalgebras of finite codimension in $W = \Der(\kk[t,\inv{t}])$), there exists $q \in \kk[t]$ such that
    $$\W(qh) \subseteq \mf{k} \subseteq \W(h).$$
    Note that $\inv{\varphi}(\W(qh)) = L(f,q(f)g)$. Hence,
    $$L(f,q(f)g) \subseteq \g.$$
    It is immediate that $R(\g) = \kk(f)$.
    
    $(2) \Rightarrow (1)$: Now suppose $R(\g)$ is a field, so that $R(\g) = F(\g)$. By Theorem \ref{thm:subfield generated by R(g)} there exists $f \in \kk[t] \setminus \{0\}$ such that $R(\g) = \kk(f)$ and $\g \subseteq L(f)$.
    
    We must show that $\g$ has finite codimension in $L(f)$. We have $f \in R(\g)$, so there exist $w\del, u\del \in \g$ such that $f = \frac{w}{u}$. Therefore, $\deg(w) - \deg(u) = \deg(f)$. Since $w\del, u\del \in \g \subseteq L(f)$, we also know that $\deg(f)$ divides $\deg(w\del)$ and $\deg(u\del)$, by Remark \ref{rem:degree of L(f,g)}. Hence, $\gcd(\deg(w\del),\deg(u\del)) = \deg(f)$. But then it follows by Lemma \ref{lem:gcd} that $\g$ contains elements of degree $k\deg(f)$ for all $k \gg 0$. Since all elements of $L(f)$ have degree $k\deg(f)$ for some $k \in \NN$, we conclude that $\g$ has finite codimension in $L(f)$, as required.
\end{proof}

\begin{rem}
    Let $\g$ be an infinite-dimensional Lie subalgebra of $\W$ and let $f \in \kk[t]$ such that $F(\g) = \kk(f)$ (which exists by Theorem \ref{thm:subfield generated by R(g)}). What was shown in the final part of the proof of Theorem \ref{thm:S is a field?} is the following surprising statement: if there exist $w\del, u\del \in \g$ such that $\frac{w}{u} = f$ then $R(\g)$ is a field.
\end{rem}

As an immediate consequence of Proposition \ref{prop:finite codimension}, Proposition \ref{prop:U(L(f,g) not noetherian}, and Theorem \ref{thm:S is a field?}, we get:
\begin{cor}\label{cor:R(g) is a field}
    Let $\g$ be an infinite-dimensional Lie subalgebra of $\W$. If $R(\g)$ is a field then $U(\g)$ is not noetherian. \qed
\end{cor}

If we do not assume that $R(\g)$ is a field, we still know that $\g \subseteq L(f)$ for some $f \in \kk[t]$, by Theorem \ref{thm:subfield generated by R(g)}. What is not clear is whether $\g$ has finite codimension in $L(f)$, or whether $U(\g)$ can be noetherian in this case.

The rest of the subsection is devoted to proving Theorem \ref{thm:subfield generated by R(g)}. Before we can prove it, we require some technical results about rational functions.

\begin{ntt}
    Let $f \in \kk(t)$ and $g \in \kk[t]$. We write $f \# g$ if there exists $h \in \kk[t]$ such that $g = fh$.
\end{ntt}

The following result shows that the relation $f \# g$ is nothing new, but it still gives a concise way of encoding the relationship between $f$ and $g$.

\begin{lem}\label{lem:numerator divides}
    Let $f = \frac{a}{b} \in \kk(t)$, where $a,b \in \kk[t]$ are coprime, and let $g \in \kk[t]$. Then $f \# g$ if and only if $a$ divides $g$. \qed
\end{lem}

In order to prove Theorem \ref{thm:subfield generated by R(g)}, we will want to consider polynomials in $f$, where $f$ is a rational function in $t$. The following proposition shows that if we have two coprime polynomials $p,q \in \kk[t]$, then $p(f)$ and $q(f)$ are still coprime after clearing the denominators.

\begin{prop}\label{prop:coprime polynomials with hats}
    Let $f = \frac{a}{b} \in \kk(t)$, where $a,b \in \kk[t]$ are coprime, and let $p,q \in \kk[t]$ be coprime. Let $n = \max\{\deg(p),\deg(q)\}$. Then
    $$b^np(f), b^nq(f) \in \kk[t]$$
    are coprime polynomials. In particular, let $m = \deg(p)$. Then $b^mp(f)$ and $b^m$ are coprime.
\end{prop}
\begin{proof}
    Without loss of generality, we may assume that $\deg(p) \geq \deg(q)$, so $n = \deg(p) \geq 1$. Let $\widehat{p} = b^np(f)$ and $\widehat{q} = b^nq(f)$. Letting $m = \deg(q)$, we have
    $$p = \prod_{i = 1}^n(t - \lambda_i), \quad q = \prod_{i = 1}^m(t - \mu_i),$$
    for some $\lambda_i, \mu_i \in \overline{\kk}$, and therefore
    $$\widehat{p} = b^np(f) = \prod_{i = 1}^n(a - \lambda_ib), \quad \widehat{q} = b^nq(f) = b^{n-m}\prod_{i = 1}^m(a - \mu_ib) \in \kk[t].$$
    We must show that $\widehat{p}$ and $\widehat{q}$ are coprime. 
    
    Suppose, for a contradiction, that $\gamma \in \overline{\kk}$ is a root of both $\widehat{p}$ and $\widehat{q}$.
    
    \noindent \textbf{Case 1:} $\gamma$ is a root of $b$.
    
    Since $\gamma$ is a root of $\widehat{p}$,
    $$a(\gamma) - \lambda_ib(\gamma) = 0$$
    for some $i$. But since $b(\gamma) = 0$, it follows that $a(\gamma) = 0$, which contradicts $a$ and $b$ being coprime.
    
    \noindent \textbf{Case 2:} $\gamma$ is not a root of $b$.
    
    As in Case 1, we must have
    \beq\label{eq:lambda}
        a(\gamma) - \lambda_ib(\gamma) = 0
    \eeq
    for some $i$. Furthermore, since $\gamma$ is a root of $\widehat{q}$ and is not a root of $b$,
    \beq\label{eq:mu}
        a(\gamma) - \mu_jb(\gamma) = 0
    \eeq
    for some $j$. Subtracting \eqref{eq:lambda} from \eqref{eq:mu},
    $$(\lambda_i - \mu_j)b(\gamma) = 0.$$
    But $b(\gamma) \neq 0$, so $\lambda_i = \mu_j$, which contradicts $p$ and $q$ being coprime.
    
    Therefore, $\widehat{p} = b^np(f)$ and $\widehat{q} = b^nq(f)$ are coprime polynomials, as required.
\end{proof}

The next result now follows immediately from Lemma \ref{lem:numerator divides} and Proposition \ref{prop:coprime polynomials with hats}.

\begin{cor}\label{cor:hat polynomial divides}
    Let $f = \frac{a}{b} \in \kk(t)$, where $a,b \in \kk[t]$ are coprime, and let $p,g \in \kk[t]$. Then $p(f) \# g$ if and only if $b^np(f) \in \kk[t]$ divides $g$, where $n = \deg(p)$.
\end{cor}
\begin{proof}
    Letting $\widehat{p} = b^np(f) \in \kk[t]$, we can write
    $$p(f) = \frac{\widehat{p}}{b^n}.$$
    We know that $\widehat{p}$ and $b^n$ are coprime by Proposition \ref{prop:coprime polynomials with hats}. By Lemma \ref{lem:numerator divides}, $p(f) \# g$ if and only if $\widehat{p}$ divides $g$.
\end{proof}

As an easy consequence of Proposition \ref{prop:coprime polynomials with hats} and Corollary \ref{cor:hat polynomial divides}, we get:

\begin{cor}\label{cor:f polynomials divide as usual}
    Let $f \in \kk(t)$, and let $p,q \in \kk[t]$ be coprime polynomials. Let $w,u \in \kk[t]$ such that
    $$\frac{w}{u} = \frac{p(f)}{q(f)}.$$
    Then $p(f) \# w$ and $q(f) \# u$.
\end{cor}
\begin{proof}
    Write $f = \frac{a}{b}$, where $a,b \in \kk[t]$ are coprime. Let $n = \max\{\deg(p), \deg(q)\}$, and let $\widehat{p} = b^np(f), \widehat{q} = b^nq(f) \in \kk[t]$. By Proposition \ref{prop:coprime polynomials with hats}, $\widehat{p}$ and $\widehat{q}$ are coprime. Furthermore, $w\widehat{q} = u\widehat{p}$. Hence, $\widehat{p}$ divides $w$ and $\widehat{q}$ divides $u$.
    
    Without loss of generality, we may assume that $n = \deg(p)$. Thus, by Corollary \ref{cor:hat polynomial divides}, $p(f) \# w$. Therefore, there exists $g \in \kk[t]$ such that $w = p(f)g$, so that
    $$up(f) = wq(f) = p(f)q(f)g.$$
    Thus $u = q(f)g$, so $q(f) \# u$.
\end{proof}

For convenience, we introduce some new notation:

\begin{ntt}
    For $f \in \kk(t)$, we write
    $$\kk[f]_\# = \{g \in \kk[t] \mid p(f) \# g \text{ for some } p \in \kk[t] \setminus \kk\}.$$
\end{ntt}

We now show that, provided $f \in \kk(t)$ has a pole at infinity, any polynomial $h \in \kk[t]$ can be written as $h = p(f)g$, where $p \in \kk[t]$, $g \in \kk[t] \setminus \kk[f]_\#$.

\begin{prop}\label{prop:f polynomial times anti f polynomial}
    Let $f \in \kk(t)$ such that $f$ has a pole at infinity, and let $h \in \kk[t]$. Then we can write $h = p(f)g$, where $p \in \kk[t], g \in \kk[t] \setminus \kk[f]_\#$.
\end{prop}
\begin{proof}
    If $h \not\in \kk[f]_\#$, then we are done, as we can take $p = 1, g = h$.
    
    If $h \in \kk[f]_\#$, then let $p_1 \in \kk[t] \setminus \kk$ such that $p_1(f) \# h$. We can write
    $$h = p_1(f)h_1,$$
    where $h_1 \in \kk[t]$. Write $f = \frac{a}{b}$, where $a,b \in \kk[t]$ are coprime. By hypothesis, $\deg(a) > \deg(b)$. Let $n_1 = \deg(p_1)$ and let $\widehat{p}_1 = b^{n_1}p_1(f) \in \kk[t]$. By Proposition \ref{prop:coprime polynomials with hats}, $\widehat{p}_1$ and $b^{n_1}$ are coprime. Furthermore,
    $$b^{n_1}h = \widehat{p}_1h_1,$$
    so $\widehat{p}_1$ divides $h$. Write $h = \widehat{p}_1\widehat{h}_1$, where $\widehat{h}_1 \in \kk[t]$. We have
    \begin{align*}
        \deg(h) &= \deg(\widehat{p}_1) + \deg(\widehat{h}_1), \\
        \deg(\widehat{p}_1) &= n_1\deg(a), \\
        \deg(h_1) &= n_1\deg(b) + \deg(\widehat{h}_1).
    \end{align*}
    Combining the above,
    $$\deg(h_1) = \deg(h) - n_1(\deg(a) - \deg(b)) < \deg(h),$$
    since $\deg(a) > \deg(b)$.
    
    Therefore, if $h \in \kk[f]_\#$ then we can write $h = p_1(f)h_1$, where $p_1, h_1 \in \kk[t]$ and $\deg(h_1) < \deg(h)$. Inducting on $\deg(h)$, we can write $h_1 = p_2(f)g$ for $p_2 \in \kk[t], g \in \kk[t] \setminus \kk[f]_\#$. Thus, we can take $p = p_1p_2$ and we are done.
\end{proof}

We are now ready to prove Theorem \ref{thm:subfield generated by R(g)}.

\begin{proof}[Proof of Theorem \ref{thm:subfield generated by R(g)}]
    We have
    $$\kk \subsetneqq K \subseteq \kk(t),$$
    so by L\"uroth's theorem there exists some $f \in \kk(t)$ such that $K = \kk(f)$. Write $f = \frac{a}{b} \in \kk(t)$, where $a, b \in \kk[t]$ are coprime. By applying a M\"obius transformation to $f$ if necessary, we may assume that $\deg(a) > \deg(b)$.
    
    Let $w\del \in \g \setminus \{0\}$. Using Proposition \ref{prop:f polynomial times anti f polynomial}, write $w = p(f)g$, where $p \in \kk[t], g \in \kk[t] \setminus \kk[f]_\#$. We claim that every element of $\g$ can be written as $q(f)g\del$ for some $q \in \kk[t]$.
    
    Let $u\del \in \g \setminus \{0\}$. Then $\frac{w}{u} = \frac{p(f)g}{u} \in \kk(f)$, so
    $$\frac{g}{u} = \frac{w/u}{p(f)} \in \kk(f).$$
    Write
    $$\frac{g}{u} = \frac{r(f)}{q(f)},$$
    where $q,r \in \kk[t]$ are coprime. By Corollary \ref{cor:f polynomials divide as usual}, $r(f) \# g$. But we know that $g \not\in \kk[f]_\#$, so $r$ must be constant. Without loss of generality, we can take $r = 1$, so that
    $$u = q(f)g \in \kk[f]g.$$
    Since $u\del \in \g$ was arbitrary, we conclude that $\g \subseteq \kk[f]g\del$.
    
    We now show that $f$ must be a polynomial. Since $\g$ is infinite-dimensional, there exists $u\del \in \g$ such that $u = q(f)g$, where $q \in \kk[t]$ is such that $n = \deg(q) > \deg(g)$. By Proposition \ref{prop:coprime polynomials with hats}, the polynomials $\widehat{q} = b^nq(f)$ and $b^n$ are coprime. We have
    $$b^nu = \widehat{q}g,$$
    and therefore $b^n$ divides $g$. But $n > \deg(g)$, so $b$ must be constant, and thus $f = \frac{a}{b} \in \kk[t]$, as claimed. Hence, $K = \kk(f)$, where $f \in \kk[t]$, and $\kk[f]g\del = L(f,g)$.
    
    We now have $\g \subseteq L(f,g)$. To conclude the proof, we must show that $f'g \in \kk[f]$, so that $L(f,g)$ is a Lie algebra by Lemma \ref{lem:infinite codimension subalgebra}, and $\g \subseteq L(f)$ by Proposition \ref{prop:L(f)}. Let $p(f)g\del, q(f)g\del \in \g \setminus \{0\}$ with $p \neq q$. Then
    $$[p(f)g\del,q(f)g\del] = r(f)f'g^2\del \in \g \setminus \{0\},$$
    where $r = pq' - p'q \in \kk[t]$. Since $F(\g) = \kk(f)$, it follows that $\frac{r(f)f'g^2}{p(f)g} \in \kk(f)$, so $f'g \in \kk(f)$.
    Therefore,
    $$f'g \in \kk(f) \cap \kk[t] = \kk[f],$$
    as required.
\end{proof}

\section{A sufficient condition for the non-noetherianity of the enveloping algebra of a subalgebra of $\W$}\label{section:condition (*)}

In this section, we adapt the methods of \cite{SierraSpenko} to give an alternative proof that $U(L(f,g))$ is not noetherian. The advantage of this method is that it is more likely to generalise to give a proof of Conjecture \ref{conj:subalgebras of W}. In particular, we give a sufficient condition, which we call condition $(*)$, for a subalgebra $\g \subseteq \W$ to have a non-noetherian enveloping algebra.

\begin{ntt}
    Let $\sigma \in \Aut(\kk[x, y])$ such that $f^\sigma(x, y) = f(x + 1, y)$. Define $T = \kk[x, y][t^{\pm 1};\sigma]$, i.e. as a vector space $$T = \bigoplus_{n \in \ZZ} \kk[x, y]t^n,$$ with $t f(x, y) = f^\sigma(x, y) t = f(x + 1, y) t$. The algebra $T$ is graded with $\deg(t) = 1$ and $\deg(x) = \deg(y) = 0$. For non-homogeneous elements $u \in T$, we write $\deg(u)$ for the highest power of $t$ which appears in the expansion of $u$.
\end{ntt}

\begin{rem}
    The algebra $T$ is isomorphic to $A_1(\kk(y))[\inv{t}]$, where $A_1(\kk(y)) = \kk(y)[t,\del]$ is the first Weyl algebra over $\kk(y)$, with commutation relation
    $$\del t - t\del = 1.$$
    Under this isomorphism, the element $x \in T$ corresponds to $-t\del \in A_1(\kk(y))[\inv{t}]$.
\end{rem}

We define a homomorphism from $W_{\geq -1}$ to $T$, first introduced by Sierra and {\v S}penko. Although they used much more general machinery to construct this map, it is straightforward to directly check that it is a homomorphism. This will allow us to work over the more accessible algebra $T$ by proving that the image of $U(L(f,g))$ in $T$ is not noetherian.

\begin{lem}\label{lem:homomorphism to T}
    The linear map
    \begin{align*}
        \varphi \colon W_{\geq -1} &\to T \\
        e_n &\mapsto - (x + ny)t^n
    \end{align*}
    induces a graded $\kk$-algebra homomorphism $\Phi \colon U(W_{\geq -1}) \rightarrow T$. This homomorphism can also be written as
    $$\Phi(f\del) = -(xf\inv{t} + y(f' - f\inv{t}))$$
    for $f \in \kk[t]$.
\end{lem}
\begin{proof}
    In order to prove that $\Phi$ is a $\kk$-algebra homomorphism, it is enough to check that $\varphi$ is a Lie algebra homomorphism, i.e. that $\varphi([e_n, e_m]) = \varphi(e_n)\varphi(e_m) - \varphi(e_m)\varphi(e_n)$. We leave this computation to the reader.
    
    It suffices to check the final sentence for $f = t^{n+1}$, so that $f\del = e_n$. Note that
    $$f' - f\inv{t} = (n + 1)t^n - t^n = nt^n,$$
    so the formula holds.
\end{proof}

Equipped with the homomorphism $\Phi$, we can now define condition $(*)$.

\begin{dfn}\label{dfn:condition (*)}
    We say that a Lie subalgebra $\g \subseteq \W$ \emph{satisfies condition $(*)$} if there exists some $q \in U(\g)$ with $\Phi(q) \in \kk[y,t,\inv{t}] \setminus \kk[y]$.
\end{dfn}

\begin{rem}
    The requirement that $\Phi(q) \not\in \kk[y]$ guarantees that $\Phi(q)$ is not central in $T$.
\end{rem}

The definition of condition $(*)$ is motivated by Sierra and {\v S}penko's proof that the enveloping algebra of the Witt algebra is not noetherian (see \cite[Section 4]{SierraSpenko}). In their paper, they define the elements
$$p_n = \frac{1}{2}(e_ne_{3n} + e_{3n}e_n) - e_{2n}^2 \in U(W),$$
and then show that
$$\Phi(p_n) = n^2y(1 - y)t^{4n} \in \kk[y,t,\inv{t}] \setminus \kk[y]$$
for all $n \in \ZZ \setminus \{0\}$. They use this to show that the two-sided ideal $\Phi(U(W))\Phi(p_n)\Phi(U(W))$ is not finitely generated as a left or right ideal of $\Phi(U(W))$.

Our goal is to generalise Sierra and {\v S}penko's proof by showing that if a subalgebra $\g \subseteq \W$ satisfies $(*)$, then $\Phi(U(\g))$ is not noetherian. We will do this by showing that the two-sided ideal $\Phi(U(\g))\Phi(q)\Phi(U(\g))$ is not finitely generated as a left ideal of $\Phi(U(\g))$.

For the rest of this section, we let $\g$ be an infinite-dimensional subalgebra of $W_{\geq -1}$ which satisfies $(*)$. Although for our purposes it suffices to consider $\Phi(q) \in y(1 - y)\kk[t]$, it is not much more difficult to work in the more general setting where $\Phi(q) \in \kk[y,t,\inv{t}] \setminus \kk[y]$.

In fact, we claim that without loss of generality, we can choose $q \in U(\g)$ so that $\Phi(q) \in t\kk[y,t] \setminus \{0\}$. Having $\Phi(q)$ in this form will significantly simplify the computations below. To prove that $q \in U(\g)$ can be chosen in this way, we write
$$\Phi(q) = \sum_{i = -m}^n f_i(y)t^i,$$
where $f_i(y) \in \kk[y]$, and let $a \in \g$ such that
$$a = \sum_{i = m+1}^{d} \alpha_i e_i,$$
where $\alpha_i \in \kk$. In other words, $a$ is some element of $\g$ with no terms of degree less than $m + 1$. Letting $q' = aq - qa \in U(\g)$, we have
\begin{align*}
    \Phi(q') &= \sum_{i = m+1}^d \sum_{j = -m}^n \alpha_i f_j(y) \Phi(e_i) t^j - \sum_{i = m+1}^d \sum_{j = -m}^n \alpha_i f_j(y) t^j \Phi(e_i) \\
    &= -\sum_{i = m+1}^d \sum_{j = -m}^n \alpha_i f_j(y) (x + iy)t^{i+j} + \sum_{i = m+1}^d \sum_{j = -m}^n \alpha_i f_j(y) t^j (x + iy)t^i \\
    &= -\sum_{i = m+1}^d \sum_{j = -m}^n \alpha_i f_j(y) (x + iy)t^{i+j} + \sum_{i = m+1}^d \sum_{j = -m}^n \alpha_i f_j(y) (x + iy + j)t^{i+j} \\
    &= \sum_{i = m+1}^d \sum_{j = -m}^n j \alpha_i f_j(y) t^{i+j} \in t\kk[y,t].
\end{align*}
We see that $\Phi(q') \in t\kk[y,t] \setminus \{0\}$. Therefore, we assume that $\Phi(q)$ is of this form.

\begin{ntt}\label{ntt:Phi(q)}
    Write
    $$\Phi(q) = f(y)\sum_{i = 1}^n g_i(y)t^i,$$
    where $f(y), g_i(y) \in \kk[y]$ are such that the $g_i(y)$ share no common factors. Let $1 \leq k \leq n$ be maximal such that $g_k(y)$ is not a multiple of $y$.
\end{ntt}

The next lemma shows that we can generate elements similar to $\Phi(q)$ of arbitrarily high degree by simply commuting $q$ with elements of $\g$.

\begin{lem}\label{lem:general arbitrary degree}
    Let $a \in \g$ and let $m = \deg(a)$. Then $\Phi(aq - qa) \in \kk[y,t]$ and $\deg(\Phi(aq - qa)) = n + m$. Furthermore, writing
    $$\Phi(aq - qa) = f(y)\sum_{i = 0}^{n+m} h_i(y)t^i,$$
    we have that $y$ does not divide $h_{m+k}(y)$.
\end{lem}
\begin{proof}
    Write $a = \sum_{i = -1}^m \alpha_i e_i$, where $\alpha_i \in \kk$ and $\alpha_m \neq 0$. Similarly to the above computation, we have
    $$\Phi(aq - qa) = f(y)\sum_{i = -1}^m \sum_{j = 1}^n j \alpha_i g_j(y) t^{i+j} \in \kk[y,t].$$
%    \begin{align*}
%        \Phi(aq - qa) &= f(y)\left(\sum_{i = -1}^m \sum_{j = 0}^n \alpha_i g_j(y) \Phi(e_i) t^j - \sum_{i = -1}^m \sum_{j = 0}^n \alpha_i g_j(y) t^j \Phi(e_i)\right) \\
%        &= f(y)\left(-\sum_{i = -1}^m \sum_{j = 0}^n \alpha_i g_j(y) (x + iy)t^{i+j} + \sum_{i = -1}^m \sum_{j = 0}^n \alpha_i g_j(y) t^j (x + iy)t^i\right) \\
%        &= f(y)\left(-\sum_{i = -1}^m \sum_{j = 0}^n \alpha_i g_j(y) (x + iy)t^{i+j} + \sum_{i = -1}^m \sum_{j = 0}^n \alpha_i g_j(y) (x + iy + j)t^{i+j}\right) \\
%        &= f(y)\sum_{i = -1}^m \sum_{j = 1}^n j \alpha_i g_j(y) t^{i+j} \in \kk[y,t].
%    \end{align*}
    We see that $\deg(\Phi(aq - qa)) = n + m$ and that
    $$h_\ell(y) = \sum_{i+j = \ell}j \alpha_i g_j(y),$$
    for all $\ell$. Letting $N = \min(n,m+k+1)$, we have
    $$h_{m+k}(y) = k \alpha_m g_k(y) + (k + 1)\alpha_{m-1}g_{k+1}(y) + \ldots + N \alpha_{m+k-N} g_N(y).$$
    By maximality of $k$, it follows that $y$ divides $g_{k+1}(y), \ldots, g_N(y)$, but does not divide $g_k(y)$. Hence, $y$ does not divide $h_{m+k}(y)$, since $k\alpha_m \neq 0$.
\end{proof}

We now show that $B = \Phi(U(\g))$ is not left noetherian, which suffices to prove that $U(\g)$ is not noetherian.

\begin{prop}
    The associative algebra $B$ is not left noetherian.
\end{prop}
\begin{proof}
    We will show that $I = B\Phi(q)B$ is not finitely generated as a left ideal of $B$. Note that $I \subseteq f(y)T$.
    
    Assume, for a contradiction, that $I$ is finitely generated as a left ideal of $B$. Let $c_1, \ldots, c_\ell \in I$ be generators of ${}_BI$. Since $\g$ is infinite-dimensional, there is some $a = e_m + \ldots \in \g$ such that $m + k > \deg(c_i)$ for all $1 \leq i \leq \ell$. Write
    $$s = \Phi(aq - qa) = f(y)\sum_{i = 0}^{n+m}h_i(y)t^i \in I = Bc_1 + \ldots + Bc_\ell$$
    as in Lemma \ref{lem:general arbitrary degree}. Then there exist $b_1, \ldots, b_\ell \in B$ such that
    $$s = b_1 c_1 + \ldots + b_\ell c_\ell.$$
    Since $m + k > \deg(c_i)$ for all $i$, it follows that if $b_j c_j$ has a nonzero contribution to the degree $m + k$ term of $b_1 c_1 + \ldots + b_\ell c_\ell$, then we must have $\deg(b_j) \geq 1$. Write $c_i = f(y)\sum_j F_{ij}(x,y)t^j$, where $F_{ij}(x,y) \in \kk[x,y]$. Note that
    \begin{align*}
        \Phi(e_\ell)c_i &= -(x + \ell y)t^\ell f(y)\sum_j F_{ij}(x,y)t^j \\
        &= -f(y)\sum_j (x + \ell y)F_{ij}(x + \ell,y)t^{j+\ell} \in f(y)(x,y)T,
    \end{align*}
    where $(x,y)$ is the ideal of $\kk[x,y]$ generated by $x$ and $y$. We therefore see that the degree $m + k$ term of $b_1 c_1 + \ldots + b_\ell c_\ell$ is contained in $f(y)(x,y)t^{m+k}$. However, Lemma \ref{lem:general arbitrary degree} says that $h_{m+k}(y) \not\in (x,y)$, a contradiction. Therefore, $B$ is not left noetherian.
\end{proof}

The following result now follows immediately.

\begin{thm}\label{thm:not noetherian}
    Let $\g$ be an infinite-dimensional subalgebra of $W_{\geq -1}$ which satisfies $(*)$. Then $U(\g)$ is not noetherian. \hfill \qed
\end{thm}

As stated in Definition \ref{dfn:condition (*)}, condition $(*)$ does not seem like an easy condition to check. The following proposition provides a method to construct elements $q \in U(\W)$ such that $\Phi(q) \in \kk[y,t,\inv{t}]$, providing a simple way to check if a subalgebra of $\W$ satisfies $(*)$.

\begin{prop}\label{prop:magical element p}
    Let $f_1,f_2,g_1,g_2 \in \kk[t]$ such that $f_1f_2 = g_1g_2$, and let $a_i = f_i\del$, $b_i = g_i\del \in W_{\geq -1}$. Let
    $$q = (a_1a_2 + a_2a_1) - (b_1b_2 + b_2b_1) \in U(W_{\geq - 1}).$$ Then $\Phi(q) \in y(1 - y)\kk[t]$. Furthermore, if $\{\deg(f_1),\deg(f_2)\} \neq \{\deg(g_1),\deg(g_2)\}$, then $\Phi(q) \neq 0$ and $\deg(\Phi(q)) = \deg(f_1) + \deg(f_2) - 2$.
\end{prop}
\begin{proof}
    Using the formula $\Phi(f\del) = -(xf\inv{t} + y(f' - f\inv{t}))$ from Lemma \ref{lem:homomorphism to T} and the commutation relation $fx = xf + f't$, we get
    \begin{align*}
        \Phi(a_1a_2 + a_2a_1) = 2&x^2f_1f_2t^{-2} + x(f_1'f_2\inv{t} - 2f_1f_2t^{-2} + f_1f_2'\inv{t}) \\
        &+ 2xy(f_1'f_2\inv{t} - 2f_1f_2t^{-2} + f_1f_2'\inv{t}) \\
        &+ y(f_1''f_2 - f_1'f_2\inv{t} + 2f_1f_2t^{-2} - f_1f_2'\inv{t} + f_1f_2'') \\
        &+ 2y^2(f_1' - f_1\inv{t})(f_2' - f_2\inv{t}),
    \end{align*}
    and similarly for $\Phi(b_1b_2 + b_2b_1)$ with $g_i$ replacing $f_i$. Since $f_1f_2 = g_1g_2$, we have $(f_1f_2)' = (g_1g_2)'$ and $(f_1f_2)'' = (g_1g_2)''$, and therefore
    \begin{align}
        f_1'f_2 + f_1f_2' &= g_1'g_2 + g_1g_2' \label{eq:first derivative}\\
        f_1''f_2 + f_1f_2'' + 2f_1'f_2' &= g_1''g_2 + g_1g_2'' + 2g_1'g_2' \label{eq:second derivative}
    \end{align}
    Hence, subtracting $\Phi(b_1b_2 + b_2b_1)$ from $\Phi(a_1a_2 + a_2a_1)$ and using \eqref{eq:first derivative}, we get
    $$\Phi(q) = y(f_1''f_2 + f_1f_2'' - g_1''g_2 - g_1g_2'') + 2y^2((f_1' - f_1\inv{t})(f_2' - f_2\inv{t}) - (g_1' - g_1\inv{t})(g_2' - g_2\inv{t})).$$
    Furthermore,
    \begin{align*}
        (f_1' - f_1\inv{t})(f_2' - f_2\inv{t}) &= f_1'f_2' - (f_1'f_2 + f_1f_2')\inv{t} + f_1f_2t^{-2}, \\
        (g_1' - g_1\inv{t})(g_2' - g_2\inv{t}) &= g_1'g_2' - (g_1'g_2 + g_1g_2')\inv{t} + g_1g_2t^{-2},
    \end{align*}
    and thus, using \eqref{eq:first derivative}, we see that
    $$(f_1' - f_1\inv{t})(f_2' - f_2\inv{t}) - (g_1' - g_1\inv{t})(g_2' - g_2\inv{t}) = f_1'f_2' - g_1'g_2'.$$
    By \eqref{eq:second derivative}, we also have
    $$f_1''f_2 + f_1f_2'' - g_1''g_2 - g_1g_2'' = -2(f_1'f_2' - g_1'g_2').$$
    Combining all the above, we conclude
    $$\Phi(q) = y(1 - y)(f_1''f_2 + f_1f_2'' - g_1''g_2 - g_1g_2'') = -2y(1 - y)(f_1'f_2' - g_1'g_2')  \in y(1 - y)\kk[t].$$
    For the final sentence, we let $n_i = \deg(f_i),m_i = \deg(g_i)$, and assume that $\{n_1,n_2\} \neq \{m_1,m_2\}$. Let $\alpha_i, \beta_i \in \kk^*$ be the leading coefficients of $f_i, g_i$, respectively. Let $\gamma = \alpha_1 \alpha_2 = \beta_1 \beta_2$ and $N = n_1 + n_2 = m_1 + m_2$. The leading term of $f_1'f_2' - g_1'g_2'$ is therefore
    $$(n_1n_2 - m_1m_2)\gamma t^{N-2} \neq 0,$$
    since $\{n_1,n_2\} \neq \{m_1,m_2\}$. Hence, $\Phi(q) \neq 0$ and $\deg(\Phi(q)) = N - 2$.
\end{proof}

In light of Proposition \ref{prop:magical element p}, we make the following definition:

\begin{dfn}
    We say that a subalgebra $\g \subseteq \W$ \emph{satisfies condition $(\dagger)$} if there exist $f_1,f_2,g_1,g_2 \in \kk[t]$ with $f_i\del,g_i\del \in \g$ such that 
    \begin{align*}
        f_1f_2 &= g_1g_2, \\
        \{\deg(f_1),\deg(f_2)\} &\neq \{\deg(g_1),\deg(g_2)\}.
    \end{align*}
\end{dfn}

By Proposition \ref{prop:magical element p}, if $\g$ satisfies $(\dagger)$ then $\g$ satisfies $(*)$, and therefore $U(\g)$ is not noetherian in this case. In particular, it is clear that $L(f,g)$ satisfies condition $(\dagger)$ for any $f,g \in \kk[t] \setminus \{0\}$ such that $f'g \in \kk[f]$ and $\deg(f) \geq 1$, providing an alternative proof of Proposition \ref{prop:U(L(f,g) not noetherian}.

In fact, every known infinite-dimensional subalgebra of $\W$ satisfies $(\dagger)$ (and therefore also satisfies $(*)$), but we have not been able to show that this will always be the case. This is the subject of ongoing research.

\end{document}